\documentclass[11pt]{amsart}
\pdfoutput=1
\usepackage[latin1]{inputenc}
\usepackage{amssymb,epsfig}
\usepackage{amsmath, amsthm}
\usepackage{a4wide}
\usepackage{amsfonts}
\usepackage{float}
\usepackage{cite}
\usepackage{pdfsync}
\usepackage{xcolor}
\usepackage{enumerate}

\usepackage{cancel}

\def\E{{\mathbb E}}

\newtheorem{theorem}{Theorem}
\newtheorem{lemma}{Lemma}

\newtheorem{remark}{Remark}

\title{Limit theorems for a minimal random walk model}
\author{Cristian F. Coletti, Renato J. Gava and Lucas R. de Lima}

\date{\today}

\address{
\newline 
UFABC - Centro de Matem\'atica, Computa\c{c}\~ao e Cogni\c{c}\~ao.
\newline
Avenida dos Estados, 5001, Santo Andr\'e - S\~ao Paulo, Brasil
\newline
e-mail:  \rm \texttt{cristian.coletti@ufabc.edu.br}
\newline
e-mail:  \rm \texttt{lucas.roberto@ufabc.edu.br}
\newline 
\newline
UFSCAR - Departamento de Estat\'{\i}stica.
\newline  Rodovia Washington Luiz, Km 235, CEP 13565-905, S\~ao Carlos, Brasil
\newline
e-mail:  \rm \texttt{gava@ufscar.br} 
}

\thanks{R.J.G. thanks FAPESP (grants 2017/10555-0 and 2018/04764-9).}
\thanks{C.F.C. thanks FAPESP grant 2017/10555-0.}
\thanks{L. R. L. thanks CAPES - Finance Code 001.}

\begin{document}


\begin{abstract}

We study the minimal random walk introduced by Kumar, Harbola and Lindenberg \cite{Harb14}.
It is a random process on $\{0, 1, \ldots  \}$ with unbounded memory which exhibits 
subdiffusive, diffusive and superdiffusive regimes. We prove the law of large numbers
for the whole parameter set. Then we prove the central limit theorem and the law of the 
iterated logarithm for the minimal random walk under diffusive and marginally superdiffusive 
behaviors. More interestingly, we establish a result for the minimal random walk when it 
possesses the three regimes; we show the convergence of its rescaled version to a non-normal random variable.

\end{abstract}

\maketitle


\section{Introduction}\label{intro}

In this work we are concerned with the minimal random walk model introduced in \cite{Harb14},
a random walk $X_n$ on $\{0,1, 2, \ldots \}$ with dependent increments and unbounded memory 
such that at each step the walker either moves to the nearest neighbour to its right or remains at the same place. The 
goal of Kumar, Harbola and Lindenberg \cite{Harb14} was to introduce a model, as simple as possible, in which 
the exact computation of the first two moments, depending on the choice of a pair of 
parameters $0 \leq q \leq 1$ and $0 \leq p \leq 1$, enables us to view the emergence of 
subdiffusive, diffusive and superdiffusive regimes.

Some years before the appearance of the minimal model proposed in \cite{Harb14}, the same authors introduced the 
lazy elephant random walk \cite{Harb10}, a generalization of the so-called elephant random walk proposed 
by Sch\"utz and Trimper \cite{ST}. This generalized model, like the one treated in this 
article, also exhibits the three different diffusion regimes; it allows the walker to go forward, 
backward or remain at rest. Therefore the random walk we deal with in this work is simpler: 
the increments have only two possible choices, either a forward step or a resting step.
All these models are  non-Markovian  since their past or memory matters 
in the future steps of the walk. In the last two decades, random walks with unbounded memory 
has received considerable attention from the physics community as one can see in \cite{Bu,daSi13,Harb10,Harb14,Harb18,K16,KM,ST} and references therein. More recently some 
mathematical papers \cite{Baur16,Bercu,CGS1,CGS2} have established rigorously the first 
limit theorems on the subject. Regarding interacting non-Markovian random walks it is worth mentioning that the authors of \cite{AR,KHG} considered the influence of memory in such systems. In both works the authors  study how  memory affects the behavior of an interacting particle system and its relation to non-Markovian exclusion processes.

Here we establish a series of results for the minimal random walk \cite{Harb14}. First, 
we prove a strong law of large numbers (Theorem \ref{lln}), then we prove a central limit 
theorem (Theorem \ref{clt}) and the law of the iterated logarithm (Theorem \ref{lil}) for 
an appropriate choice of parameters at which the model shows diffusive and 
marginally superdiffusive behaviors. Finally, in Theorem \ref{q0thm}, we state our main 
result. We address the case in which $q=0$ where the random walk $X_n$ exhibits subdiffusion, 
diffusion and superdiffusion regimes depending on whether $p < 1/2$, $p=1/2$ or $p > 1/2$ 
respectively. We show that the rescaled random walk $X_n$ converges to a non-degenerate 
non-normal random variable for any $p \in (1/2,1)$. In order to show that it does not obey 
a normal distribution we compute the first four moments of the limit random variable. It 
remains open to find out the correct scaling of $X_n$ in order to obtain convergence for 
$q = 0$ and $p \leq 1/2$. 
The proofs given here are based on a martingale approach \cite{HH,St2,St1}.
The martingale theory is an extension of the theory of stochastic processes with independent 
increments, allowing dependency among them.
Given a sequence $\{ M_n \}_{n \geq 0}$ and all the information $\mathcal{F}_n$ on this sequence 
up to time $n$, we say that $M_n$ is a martingale if 
\begin{align*}
\mathbb{E}(M_{n+1}| \mathcal{F}_n) = M_n.
\end{align*}
In other words, conditioned on the history up to time $n$, the average value at $M_{n+1}$ is equal
to $M_n$. A nice introduction to the subject can be found in Williams \cite{Wi}.


The paper is organized as follows. In the next section we define the model formally. In Section 
\ref{main} we state the main results of this work. We begin Section \ref{proofs} with some preliminaries results and then we present the proofs of  the main theorems.

\section{Definition of the model}

We now define the  minimal random walk as follows. It starts at $X_0 = 0$. At each discrete time step the 
walker moves one step to the right or stays put. Therefore
\begin{equation*}
X_{n+1} = X_n + \eta_{n+1},
\end{equation*}
where $\eta_{n+1} \in \{0,1 \}$ is a random variable. The memory consists of the set of random variables 
$\eta_{n^{\prime}}$ at previous time steps which the walker remembers as follows.

Initially, the walker jumps to the right with probability $s$ and  remains at 0 with probability $1-s$.
At time $n+1$, for $n\geq 1$, a number $n^{\prime}$ is chosen from the set $\{1,2, \ldots , n\}$
uniformly at random. Then $\eta_{n+1}$ is determined stochastically by the following rule.
If $\eta_{n^{\prime}} =  1$, then
\begin{align*}
\eta_{n+1} = \left\{\begin{array}{cl}
1 & \text{ with probability } \, p \\
0 &  \text{ with probability } 1-p
\end{array}\right.
.
\end{align*}
If $\eta_{n^{\prime}} = 0$, then 
\begin{align*}
\eta_{n+1} = \left\{\begin{array}{cl}
1 & \text{ with probability } \, q \\
0 &  \text{ with probability } 1-q
\end{array}\right.
.
\end{align*}

It turns out from the definition that $\displaystyle X_n = \sum_{k=1}^n \eta_k$ and that
\begin{align}\label{conditional}
\mathbb{P}[\eta_{n+1} = \eta|\eta_1, \ldots, \eta_n] & = \frac{1}{n} \sum_{k=1}^n 
\left[1 - \eta +\left(2\eta -1\right)( q + \alpha \eta_k )\right] \\
& = 1 - \eta +\left(2\eta -1\right)( q + \alpha \frac{X_n}{n}) \quad \mbox{for} \ n \geq 1,  \nonumber
\end{align}
where $\eta \in \{0,1 \}$ and $\alpha = p -q$. 

The expectation of the increment $\eta_{n+1}$ is given by
\begin{equation}\label{conditional2} 
\mathbb{E}[\eta_{n+1}] = q + \alpha \frac{\mathbb{E}[X_n]}{n} \quad \mbox{for} \ n \geq 1.
\end{equation}
In \cite{Harb14} the authors showed that
\begin{align} \label{mean}
\mathbb{E}[X_n] & = \frac{qn}{1- \alpha} + (s-\frac{q}{1 - \alpha}) \frac{\Gamma(n+\alpha)}{\Gamma(1+ \alpha)\Gamma(n)} \\
& \sim \frac{qn}{1- \alpha} + (s-\frac{q}{1 - \alpha})\frac{n^{\alpha}}{\Gamma(1+ \alpha)}
\quad \mbox{ for } \quad n \gg 1. \nonumber
\end{align}



\section{Main results}\label{main}

\begin{theorem}\label{lln}
Let $(X_n)_{n \geq1}$ be the minimal random walk described above. Then
\begin{eqnarray}
\lim_{n \rightarrow \infty} \frac{X_n - \mathbb{E}[X_n]}{n} = 0  \quad  \mbox{a.s.} \nonumber
\end{eqnarray}
for any value of $\alpha \in [-1,1)$. In other words,
\begin{eqnarray}
\lim_{n \rightarrow \infty} \frac{X_n}{n} = \frac{q}{1 - \alpha} \quad \mbox{a.s.} \nonumber
\end{eqnarray}
\end{theorem}

\begin{remark}
The case $\alpha  = 1$ is not covered by the strong law of large numbers(SLLN). Indeed, when $p = 1$ and $q = 0$, the walk is trivial since by definition its dynamics is determined by the first step $\eta_1$, that is, $\eta_n = \eta_1$ 
for all n $\geq 1$. Hence $X_n /n$ reduces to a binary random variable. 
\end{remark}


\begin{theorem}\label{clt}
Let $(X_n)_{n \geq 1}$ be the minimal random walk and let $\alpha \leq 1/2, p<1$ and $q>0$. 
\begin{itemize}
\item[a)] If $\alpha < 1/2$, then 
\begin{align}
\dfrac{X_n - \dfrac{q}{1 - \alpha}n}{\sqrt{n}} \xrightarrow{d} N\left( 0,\frac{q(1-p)}{(1 - \alpha)^2(1 - 2 \alpha)} \right). \nonumber
\end{align}
\item[b)] If $\alpha = 1/2$, then
\begin{align}
 \dfrac{X_n - 2qn}{\sqrt{n \log n}} \xrightarrow{d} N\left( 0,4q(1-p)\right) . \nonumber
\end{align}
\end{itemize}
\end{theorem}

\begin{theorem}\label{lil}
Let $(X_n)_{n \geq 1}$ be the minimal random walk and let $\alpha \leq 1/2$, $p<1$ and $q > 0$.
\begin{itemize}
\item[a)] If $\alpha < 1/2$, then 
\begin{align}
\limsup_{n \to \infty}  \dfrac{\vert X_n - \dfrac{qn}{1- \alpha} \vert }{\sqrt{2 n \log \log n}} = \sqrt{\frac{q(1- p)}{(1 - \alpha)^2(1 - 2\alpha)}} \mbox{ a.s. } \nonumber
\end{align}
 
\item[b)] If $\alpha = 1/2$, then
\begin{align}
\limsup_{n \to \infty}  \dfrac{|X_n - 2qn |}{\sqrt{2 n \log n \log \log \log n}} = \sqrt{4q(1 - p)} \mbox{ a.s. } \nonumber
\end{align}
\end{itemize}
\end{theorem}

We now state our main result. It deals with the case $q =0$ and 
$1/2 < p < 1$ where the random walk exhibits a superdiffusion regime.

\begin{theorem}\label{q0thm}
Let $(X_n)_{n \geq 1}$ be the minimal random walk. If $ q = 0$ and $ 1/2 <p<1$, then
\begin{eqnarray}
\frac{X_n}{n^{p} \Gamma(1+ p)^{-1}} - s \to M \text{ a.s. and in }  L^m  \text{ for } m \geq 1 \nonumber, 
\end{eqnarray}
where $M$ is a non-normal random variable such that 
\begin{align*}
\mathbb{E}(M) & = 0\\
\mathbb{E}(M^2) & =  \frac{2s \Gamma(1+p)^2}{\Gamma(1+ 2p)} - s^2  \\
\mathbb{E}(M^3) & = \frac{6s \Gamma(1+p)^3}{\Gamma(1+ 3p)} - \frac{6s^2 \Gamma(1+p)^2}{\Gamma(1+ 2p)} + 2s^3 \\
\mathbb{E}(M^4) & =  \frac{24s \Gamma(1+p)^4}{\Gamma(1+ 4p)} - \frac{24s^2 \Gamma(1+p)^3}{\Gamma(1+ 3p)} + \frac{12s^3 \Gamma(1+p)^2}{\Gamma(1+ 2p)} - 3s^4.
\end{align*}
\end{theorem}


\begin{remark}
We note that the same approach employed in the proof of Theorem \ref{q0thm} works when 
the parameters are $q > 0$ and $\alpha > 1/2$ and the model exhibits superdiffusion 
behavior. This occurs, basically, because the variance of the martingale $M_n$, which 
is related to $X_n$ and is introduced in Subsection \ref{pre}, converges as $n \to \infty$. 
\end{remark}

\section{Proofs}\label{proofs}

\subsection{Preliminaries}\label{pre}

We know from \eqref{conditional} that  
\begin{align}\label{cond3}
\mathbb{P}[\eta_{n+1} = 1|\eta_1, \ldots, \eta_n] = q + \alpha \frac{X_n}{n}.
\end{align}
Put 
\begin{align*}
a_1 = 1 \quad  \text{ and } \quad a_n = \prod_{j=1}^{n-1} \left(1+\frac{\alpha}{j}\right) \, \text{ for } n \geq 2 \, \text{ and } \, \alpha > -1.
\end{align*} 
Note that 
\begin{align}\label{approx}
a_{n} = \frac{\Gamma(n + \alpha )}{\Gamma(n)\Gamma(\alpha +1)} \sim \frac{n^{\alpha}}{\Gamma(1 + \alpha)} \, \text{ as } n \to \infty.
\end{align}
and that $a_n \to \infty$ as $n \to \infty$ if $\alpha  > 0$, $a_n =1$ for $n \geq 1$ if $\alpha = 0$, and $a_n \to 0 $ if 
$\alpha < 0$.

Define the filtration $\mathcal{F}_n=\sigma(\eta_1, \ldots, \eta_n)$ and $M_n = \frac{X_n - \mathbb{E}[X_n]}{a_n}$ for $n \geq 1$. 
We claim that $\{ M_n \}_{n \geq 1}$ is a martingale with respect to $\{ \mathcal{F}_n \}_{n \geq 1}$, for
\begin{eqnarray}
\mathbb{E}[M_{n+1}| \mathcal{F}_n] &=& \frac{(X_n - \mathbb{E}[X_n])}{a_{n+1}} + \frac{\mathbb{E}[\eta_{n+1}|\mathcal{F}_n] - \mathbb{E}[\eta_{n+1}]}{a_{n+1}} \nonumber \\
&=& \frac{(X_n - \mathbb{E}[X_n])}{a_{n+1}} + \frac{\alpha \frac{X_n}{n} -\alpha \frac{\mathbb{E}[X_n]}{n}}{a_{n+1}} \nonumber \\
&=& \frac{(X_n - \mathbb{E}[X_n])}{a_{n+1}} + \frac{\frac{\alpha}{n}(X_n - \mathbb{E}[X_n])}{a_{n+1}} \nonumber \\
&=& (X_n - \mathbb{E}[X_n]) \frac{\left(1+\frac{\alpha}{n}\right)}{a_{n+1}} \nonumber \\
&=& M_n .\nonumber
\end{eqnarray}
Note that \eqref{cond3} amounts to $\mathbb{P}(\eta_{n+1} = 1| \mathcal{F}_n) = q + \alpha \frac{X_n}{n}$.

The walker jumps to the right at the $n$-th step time with probability $p_n = \mathbb{E}[\eta_n]$. Then combining (\ref{conditional}) and (\ref{mean}), we have that
\begin{equation} \label{pn}
p_n = \frac{q}{1-\alpha} + \alpha\frac{a_{n-1}}{n-1}\left( s-\frac{q}{1-\alpha} \right) \text{ for }\alpha< 1.
\end{equation}

A direct consequence of \eqref{approx} is the following result.
\begin{lemma}\label{diverges}
The series $\displaystyle \sum_{n=1}^{\infty} \frac{1}{a^2_n}$ converges if and only if $\alpha > 1/2$.
\end{lemma}

Let $(D_n)_{n \geq1}$ be the martingale differences defined by $D_1 = M_1$ and, for $n \geq 2, D_n = M_n - M_{n-1}$.
Observe that
\begin{eqnarray}\label{Dj}
D_j &=&  \frac{X_j - \mathbb{E}[X_j]}{a_j} - \frac{X_{j-1} - \mathbb{E}[X_{j-1}]}{a_{j-1}} \nonumber \\
&=& \frac{\eta_j - \mathbb{E}[\eta_j]}{a_j} - \frac{\left(X_{j-1} - \mathbb{E}[X_{j-1}]\right)}{j-1} \frac{\alpha}{a_j} . 
\end{eqnarray}
Furthermore, since the increments $\eta_j$'s are uniformly bounded, it is trivial to see that 
\begin{align}\label{2aj}
|D_j| \leq \frac{2}{a_j}  \quad  \mbox{a.s.}
\end{align}

Let us state a well known result involving the gamma function and which will be used afterwards.
\begin{lemma}\label{gammalemma}
For any non-negative real numbers $a$ and $b$ such that $b \neq a +1$ and for all $n \geq 1$, 
we have
\begin{align}\label{gammasum}
\sum_{k=1}^{n} \frac{\Gamma(k + a)}{\Gamma(k + b)} = \frac{\Gamma(n + a +1)}{(b - a -1)\Gamma(n +b)}
\left( \frac{\Gamma(n+b) \Gamma(1+ a)}{\Gamma(n + a +1)\Gamma(b)} - 1 \right).
\end{align}
\end{lemma}

\subsection{Auxiliary results for martingale differences}

In the course of the proofs of the limit theorems we will use some known results from 
martingale theory. We state the following theorems and their proofs can be found in the 
cited references. Let $(M_n, \mathcal{F}_n)_{n \geq1}$ be a given martingale and let 
$(D_n)_{n \geq1}$ be its associated sequence of martingale differences.

\begin{theorem}{\cite[Theorem~3.3.1]{St1}}\label{St1.thm.3.3.1}
Let $b_j$ be $\mathcal{F}_{j-1}$-measurable for each $j \geq 1$ such that $0 < b_j \nearrow \infty ~~\mbox{a.s.}$ If
\[\sum_{j=1}^\infty \frac{\E \left[ |D_j|^r | \mathcal{F}_{j-1} \right]}{b_j^r} < \infty \quad \mbox{for some } 0 < r \leq 2,\]
then
\[{M_n}/{b_n} \to 0 \quad \mbox{a.s.}\]
\end{theorem}

\begin{theorem}{\cite[Corollary~3.1]{HH}}\label{HH.cor.3.1}
Let $\left\lbrace \sum_{k=1}^i D_{nk}, \mathcal{F}_{ni}, 1 \leq i \leq k_n , n \geq 1 \right\rbrace$ be a martingale array with $\mathcal{F}_{ni} \subset \mathcal{F}_{(n+1)i}$ , $\sum_{k=1}^i D_{nk} \in L^2$ and $\mathbb{E}\left[\sum_{k=1}^i D_{nk}\right]=0$ for all $i \in \{1, 2, \dots, k_n\}$, $n \geq 1$. Let $\sigma^2$ be an almost surely finite random variable. If

\begin{enumerate}[a)]
    \item  for all $\varepsilon >0$, $\sum_{j = 1}^{k_n}\mathbb{E}(D^{2}_{nj} \mathbb{I} (|D_{nj}| > \varepsilon) | \mathcal{F}_{j-1}) \xrightarrow{p} 0$ (conditional Lindeberg condition),  where $\mathbb{I}(A)$ is the indicator function
    
    \item $V_{n k_n}^2 =\sum_{j = 1}^{k_n}\mathbb{E}(D^{2}_{nj} | \mathcal{F}_{j-1}) \xrightarrow{p} \sigma^2$
\end{enumerate}

then
\[ \sum_{k=1}^{k_n} D _{nk} \xrightarrow{d} Z\]
where the characteristic function of $Z$ is given by $\phi_Z(t) = \E\left[ \exp\left({-\frac{1}{2}\sigma^2t^2}\right) \right]$. In particular, if $\sigma^2$ is almost surely a constant, $Z \stackrel{d}{=} N(0, \sigma^2)$.
\end{theorem}

\begin{theorem}{\cite[Theorems~1-2]{St2}}\label{St2.thm.1-2}
Consider a sequence of positive constants $(s_n)_{n \geq 1}$ such that $\lim\limits_{n\to\infty}s_n = \infty$ and
\[\frac{1}{s_n^2} \displaystyle \sum_{j=1}^n \mathbb{E}[D_j^2 | \mathcal{F}_{j-1}] \xrightarrow{a.s} 1.\]
Define $u_n = \sqrt{2\log\log s_n^2}$. If there is a sequence of random variables $(K_n)_{n\geq 1}$ such that
\[|D_j| \leq K_j \frac{s_j}{u_j} \quad \mbox{a.s.}\]
where $\lim\limits_{n\to\infty}K_n = 0$ and each $K_n$ is $\mathcal{F}_{n-1}$-measurable, then
\[\limsup_{n\to\infty}\frac{M_n}{s_n u_n} = 1 \quad \mbox{a.s.}\]
\end{theorem}

The last result which we shall need in what follows is the following:
\begin{theorem}{ \cite[Theorem 2.10]{HH}}\label{burkineq} 
If $(M_n, \mathcal{F}_n)_{n \geq1}$ is a martingale and $1 < m < \infty$, then there exists constants
$c_1$ and $c_2$ depending only on $m$ such that 

\begin{align*}
c_1 \,  \mathbb{E}(|\sum_{j = 1}^{n}D_j^{2}|^{m/2}) \leq 
\mathbb{E}(|M_n|^{m}) \leq c_2 \,  \mathbb{E}(|\sum_{j = 1}^{n}D_j^{2}|^{m/2}).
\end{align*}
\end{theorem}

\subsection{Strong law of large numbers}
\begin{proof}[Proof of Theorem \ref{lln}]
The result is a straightforward consequence of Theorem \ref{St1.thm.3.3.1}. 
Define $b_i = i^{1 - \alpha}$
and take $r \in (1,2)$. 
Now use inequality \eqref{2aj} to observe that the conditions of the theorem are fulfilled. Thus we arrive at the desired conclusion.
\end{proof}

\subsection{Central limit theorem}

\begin{proof}[Proof of Theorem \ref{clt}]
Define 
\begin{align}\label{s2}
s^2_n := \sum_{j=1}^n \frac{p_j(1 - p_j)}{a^2_j} \, \text{ for } n \geq 1.
\end{align}
Note that Theorem \ref{lln} gives $\dfrac{X_n - \E[X_n]}{n} = \mbox{o}(1)$ (a.s.). Now equations \eqref{conditional} and \eqref{Dj} and the fact that $\eta_j^2 = \eta_j ~\mbox{a.s.}$ allow
us to conclude that
\begin{eqnarray}\label{d2}
\mathbb{E}\left[\left. D_j^2 \right| \mathcal{F}_{j-1} \right]
& = & \frac{1}{a_j^2}\mathbb{E}\left[\left. \left(\eta_j - \mathbb{E}[\eta_j] - \frac{X_{j-1} - \mathbb{E}[X_{j-1}]}{j-1}\alpha \right)^2 \right| \mathcal{F}_{j-1} \right] \nonumber \\
& = & \frac{1}{a_j^2}\mathbb{E}\left[\left. \eta_j^2-2p_j\eta_j +p_j^2\right| \mathcal{F}_{j-1} \right] + \mbox{o}\left(\frac{1}{a_j^2}\right) \nonumber \\
& = & \frac{p_j(1-p_j)}{a_j^2} + \mbox{o}\left(\frac{1}{a_j^2}\right) \text{ }\text{ a.s.}
\end{eqnarray}

Both claims in Theorem \ref{clt} amount to prove that  
\begin{eqnarray}\label{ansn}
\frac{X_n - \mathbb{E}[X_n]}{a_n s_n} \xrightarrow{d} N(0,1).
\end{eqnarray}
Recall that $\mathbb{E}(X_n) \sim \frac{qn}{1- \alpha}$ and note that \eqref{conditional2} implies 
\begin{align*}
\mathbb{P}(\eta_n = 1) \to \frac{q}{1 - \alpha}  \quad \mbox{ as } n \to \infty.
\end{align*}
Combining Lemma \ref{diverges} and (\ref{approx}) we get
\begin{align}
& s_n^{2} \sim  \frac{q(1-p)}{(1 - \alpha)^2} \Gamma(1+ \alpha)^{2} \frac{n^{1- 2\alpha}}{1- 2\alpha}, \text{ if } \alpha < 1/2  \label{sn1},\\
& s_n^{2} \sim 4 q(1-p)\Gamma(3/2)^{2} \log n, \text{ if } \alpha =1/2 \label{sn2}
\end{align}
which in turn implies that 
\begin{align}
& a_n s_n \sim \sqrt{\dfrac{q(1-p)n}{(1- \alpha)^2(1 - 2 \alpha)}}, \, \text{ if } \alpha < 1/2 \label{ansn1} \\
& a_n s_n \sim 2\sqrt{q(1 - p)n \log n}, \, \text{ if } \alpha =1/2. \label{ansn2}
\end{align}

We now check (\ref{ansn}). Let us verify that the two conditions of Theorem \ref{HH.cor.3.1} are fulfilled. Start with the conditional Lindeberg condition. Let 
$D_{nj} = \frac{D_j}{s_n}$ for $1 \leq j \leq n$. Given $\varepsilon >0$, we need to prove that
\begin{align}\label{tozero}
\sum_{j = 1}^{n}\mathbb{E}(D^{2}_{nj} \mathbb{I} (|D_{nj}| > \varepsilon) | \mathcal{F}_{j-1}) \to 0  \text{ a.s.}
\end{align}
If $\alpha  \geq 0$, then $a_n \geq 1$ and $s_n \to \infty$. Next note that 
$\lbrace |D_{nj}| > \varepsilon \rbrace \subset \lbrace \frac{2}{s_n} > \varepsilon \rbrace \mbox{ a.s}$,
but the last set is a.s. empty for $n$ large enough, so (\ref{tozero}) holds.
If $\alpha < 0 $, observe that $\lim_{n \to \infty }a_n \, s_n  = \infty$ and 
$a_j^{-1} \leq a_n^{-1}$ for $j= 1, \ldots, n$. Then it is easy to see that 
$\lbrace |D_{nj}| > \varepsilon \rbrace \subset \lbrace \frac{2}{a_n s_n} > \varepsilon \rbrace \mbox{ a.s.}$
and again the latter set is a.s. empty for $n$ sufficiently large.

Lemma \ref{diverges} and \eqref{s2} yield that $s_n \to \infty$ as $n \to \infty$ if and only 
if $\alpha \leq 1/2$. In view of this fact, (\ref{s2}), \eqref{d2} and Theorem \ref{lln} we may 
claim that
\begin{eqnarray}\label{as}
\frac{1}{s_n^2} \displaystyle \sum_{j=1}^n \mathbb{E}[D_j^2 | \mathcal{F}_{j-1}] \xrightarrow{a.s} 1.
\end{eqnarray}
In virtue of Theorem \ref{HH.cor.3.1}
we may conclude that
\begin{eqnarray}
\displaystyle \sum_{j=1}^n D_{nj} = \frac{X_n - \mathbb{E}[X_n]}{a_n s_n} \xrightarrow{d} N(0,1). \nonumber
\end{eqnarray}

\end{proof}

\subsection{Law of the iterated logarithm}

\begin{proof}[Proof of Theorem \ref{lil}]
Both claims in Theorem \ref{lil} follow from a straightforward application of Theorem \ref{St2.thm.1-2} to our random walk. Consider $K_n:= \frac{2 u_n}{a_n s_n}$; then we can verify that $\lim\limits_{n \to \infty} K_n = 0$ as a consequence of (\ref{sn1}), (\ref{sn2}), (\ref{ansn1}) and (\ref{ansn2}). Therefore, the proof of the result in item a) of Theorem \ref{lil} follows from \eqref{2aj}, \eqref{sn1} and \eqref{as} 
and the result in item b) of the same theorem follows from \eqref{2aj}, \eqref{sn2} and \eqref{as}.
\end{proof}

\subsection{Proof of Theorem \ref{q0thm}}

Notice that now
\begin{align}\label{condq0}
\mathbb{P}[\eta_{n+1} = 1|\mathcal{F}_n] = p \frac{X_n}{n} \quad \mbox{for} \ n \geq 1.
\end{align}
Equations \eqref{conditional2} and \eqref{mean} imply that $\mathbb{P}(\eta_n = 1) \to 0  \mbox{ as } n \to \infty$.
Moreover, we get from \eqref{mean} and \eqref{approx} that 
\begin{align*}
\mathbb{E}(X_n) =  sa_n \sim s\frac{n^{p}}{\Gamma(1+ p)}.
\end{align*}

As part of the proof of Theorem \ref{q0thm}, we need to compute the second, the third and 
the fourth moments of $X_n$.

\subsubsection{Computation of $\mathbb{E}(X_n^2)$, $\mathbb{E}(X_n^3)$ and $\mathbb{E}(X_n^4)$}


Let us start with $\mathbb{E}(X_n^2)$. Keep in mind that $\eta_j^2 = \eta_j ~\mbox{a.s.}$, so from \eqref{condq0} we obtain
\begin{align*}
\mathbb{E}(X_{n+1}^2| \mathcal{F}_n) = X_n^2\left( \frac{n+2p}{n} \right)  + \frac{pX_n}{n}, 
\end{align*}
which yields 
\begin{align*}
\mathbb{E}(X_{n+1}^2) = \mathbb{E}(X_n^2)\left( \frac{n+2p}{n} \right) + p\frac{\mathbb{E}(X_n)}{n}.
\end{align*}
Next we obtain by induction that
\begin{align*}
\mathbb{E}(X_n^2) =   \frac{s \Gamma(n +2p)}{\Gamma(n)\Gamma(1+ 2p)} \, + \,
\frac{sp \Gamma(n +2p)}{\Gamma(n)\Gamma(1+ p)} \sum_{k = 1}^{n-1}\frac{\Gamma(k + p)}{\Gamma(k +1 + 2p)}.
\end{align*}
Using Lemma \ref{gammalemma} we obtain the exact and the asymptotic formulas 
\begin{align} \label{2m}
\mathbb{E}(X_n^2) & = \frac{2s \Gamma(n +2p)}{\Gamma(n)\Gamma(1+ 2p)} \, - \,
\frac{s \Gamma(n +p)}{\Gamma(n)\Gamma(1+ p)}. \nonumber \\
&  \sim  \frac{2s n^{2p}}{\Gamma(1+ 2p)}
\end{align}

Let us apply the same ideas to compute $\mathbb{E}(X_n^3)$.
It follows from \eqref{condq0} that 
\begin{align*}
\mathbb{E}(X_{n+1}^3| \mathcal{F}_n) = X_n^3\left( \frac{n+3p}{n} \right)  + 
3p\frac{X_n^2}{n} + p\frac{X_n}{n}, 
\end{align*}
which yields
\begin{align*}
\mathbb{E}(X_{n+1}^3) = \mathbb{E}(X_n^3)\left( \frac{n+3p}{n} \right)  + 
3p\frac{\mathbb{E}(X_n^2)}{n} + p\frac{\mathbb{E}(X_n)}{n}.
\end{align*}
Induction gives us the following expression
\begin{align*}
\mathbb{E}(X_n^3) =   & \frac{s \Gamma(n +3p)}{\Gamma(n)\Gamma(1+ 3p)} \, + \,
\frac{6sp \Gamma(n +3p)}{\Gamma(n)\Gamma(1+ 2p)} \sum_{k = 1}^{n-1}\frac{\Gamma(k + 2p)}{\Gamma(k +1 + 3p)}
 \\
& - \, \frac{2sp \Gamma(n +3p)}{\Gamma(n)\Gamma(1+ p)} \sum_{k = 1}^{n-1}\frac{\Gamma(k + p)}{\Gamma(k +1 + 3p)}.
\end{align*}
Apply Lemma \ref{gammalemma} to the middle and right terms above to get 
\begin{align} \label{3m}
\mathbb{E}(X_n^3) & = \frac{s \Gamma(n +3p)}{\Gamma(n)\Gamma(1+ 3p)} \, + \, 
\left( \frac{6s \Gamma(n +3p)}{\Gamma(n)\Gamma(1+ 3p)} - \frac{6s \Gamma(n +2p)}{\Gamma(n)\Gamma(1+ 2p)} \right)
 \nonumber \\
& \,  - \,  \left( \frac{s \Gamma(n +3p)}{\Gamma(n)\Gamma(1+ 3p)} - \frac{s \Gamma(n +p)}{\Gamma(n)\Gamma(1+ p)} \right)
\nonumber \\
& = \frac{6s \Gamma(n +3p)}{\Gamma(n)\Gamma(1+ 3p)} - \frac{6s \Gamma(n +2p)}{\Gamma(n)\Gamma(1+ 2p)}
+ \frac{s \Gamma(n +p)}{\Gamma(n)\Gamma(1+ p)} \nonumber \\
& \sim \frac{6s n^{3p}}{\Gamma(1+ 3p)}.
\end{align}

Repeating all over again to $\mathbb{E}(X_n^4)$, we obtain
\begin{align}\label{4m}
\mathbb{E}(X_n^4) & =   \frac{s \Gamma(n +4p)}{\Gamma(n)\Gamma(1+ 4p)} \, + \,
\frac{36sp \Gamma(n +4p)}{\Gamma(n)\Gamma(1+ 3p)} \sum_{k = 1}^{n-1}\frac{\Gamma(k + 3p)}{\Gamma(k +1 + 4p)}
  \nonumber \\
& - \, \frac{28sp \Gamma(n +4p)}{\Gamma(n)\Gamma(1+ 2p)} \sum_{k = 1}^{n-1}\frac{\Gamma(k + 2p)}{\Gamma(k +1 + 4p)}
\, + \,
\frac{3sp \Gamma(n +4p)}{\Gamma(n)\Gamma(1+ p)} \sum_{k = 1}^{n-1}\frac{\Gamma(k + p)}{\Gamma(k +1 + 4p)}  \nonumber \\
& =   \frac{s \Gamma(n +4p)}{\Gamma(n)\Gamma(1+ 4p)} \, + \, \left( \frac{36s \Gamma(n +4p)}{\Gamma(n)\Gamma(1+ 4p)} \, - \, \frac{36s \Gamma(n +3p)}{\Gamma(n)\Gamma(1+ 3p)} \right) \nonumber \\
& - \, \left( \frac{14s \Gamma(n +4p)}{\Gamma(n)\Gamma(1+ 4p)} \, - \, \frac{14s \Gamma(n +2p)}{\Gamma(n)\Gamma(1+ 2p)} \right) \, + \, 
\left( \frac{s \Gamma(n +4p)}{\Gamma(n)\Gamma(1+ 4p)} \, - \, \frac{s \Gamma(n +p)}{\Gamma(n)\Gamma(1+ p)} \right) \nonumber \\
& = \frac{24s \Gamma(n +4p)}{\Gamma(n)\Gamma(1+ 4p)} - \frac{36s \Gamma(n +3p)}{\Gamma(n)\Gamma(1+ 3p)}
+ \frac{14s \Gamma(n +2p)}{\Gamma(n)\Gamma(1+ 2p)} - \frac{s \Gamma(n +p)}{\Gamma(n)\Gamma(1+ p)} 
\nonumber  \\
& \sim \frac{24s n^{4p}}{\Gamma(1+ 4p)}.
\end{align}


\subsubsection{Almost sure convergence}

\begin{proof}[Proof of Theorem \ref{q0thm}]

We begin with our martingale difference. Recall that $|D_j| \leq \frac{2}{a_j}$ (see \eqref{2aj}). Then apply the Burkholder's inequality (see Theorem \ref{burkineq} above) for $m > 1$ to get
\begin{align}\label{burk}
\mathbb{E}(|M_n|^{m}) \leq c \,  \mathbb{E}(|\sum_{j = 1}^{n}D_j^{2}|^{m/2}),
\end{align}
where $c$ is a positive constant depending on $m$ only. 
Along with Lemma \ref{diverges} for $ p > 1/2$ the previous inequality implies that 
$\sup \mathbb{E}(|M_n|^{m}) < \infty$, and it follows from the $L^m$ convergence theorem for 
martingales that $M_n \to M$ a.s. and in $L^m$. Moreover,
\begin{align*}
|\mathbb{E}(M)| = |\mathbb{E}(M - M_n)| \leq \mathbb{E}(|M - M_n|) \leq \mathbb{E}(|M - M_n|^{2})^{1/2} \to 0 \, \text{ as } n \to \infty.
\end{align*}
By \eqref{2m} we note that 
\begin{align*}
\mathbb{E}(M^2) & = \lim_{n \to \infty} \mathbb{E}((\frac{X_n}{a_n})^2) -s^2
=  \frac{2s \Gamma(1+p)^2}{\Gamma(1+ 2p)} - s^2,
\end{align*}
where $a_n \sim n^{p }/\Gamma(1 + p)$.
In order to arrive  at the conclusion of the theorem, let us compute the third and fourth 
moment of $M$ and verify that they do not correspond to the third and fourth moment of a normally distributed random variable. Applying \eqref{approx}, 
\eqref{2m}, \eqref{3m} and \eqref{4m}, we obtain 
\begin{align*}
\mathbb{E}(M^3) & = \lim_{n \to \infty} \big\{ \mathbb{E}((\frac{X_n}{a_n})^3) -3s \mathbb{E}((\frac{X_n}{a_n})^2) 
+ 3s^2\mathbb{E}(\frac{X_n}{a_n}) - s^3 \big\}  \\
& =  \frac{6s \Gamma(1+p)^3}{\Gamma(1+ 3p)} - \frac{6s^2 \Gamma(1+p)^2}{\Gamma(1+ 2p)} + 2s^3
\end{align*}
and
\begin{align*}
\mathbb{E}(M^4) & = \lim_{n \to \infty} \big\{ \mathbb{E}((\frac{X_n}{a_n})^4) - 4s \mathbb{E}((\frac{X_n}{a_n})^3) + 6s^2 \mathbb{E}((\frac{X_n}{a_n})^2) - 4s^3\mathbb{E}(\frac{X_n}{a_n}) + s^4 \big\}  \\
& =  \frac{24s \Gamma(1+p)^4}{\Gamma(1+ 4p)} - \frac{24s^2 \Gamma(1+p)^3}{\Gamma(1+ 3p)} + \frac{12s^3 \Gamma(1+p)^2}{\Gamma(1+ 2p)} - 3s^4,
\end{align*}
which completes the proof of Theorem \ref{q0thm}.

\end{proof}

\end{document}